\documentclass[12pt]{article}

\usepackage[utf8]{inputenc}

\usepackage{mathtools}
\usepackage{amsfonts}
\usepackage{amsthm}
\usepackage{bbm}

\newtheorem{lemma}{Lemma}
\newtheorem{theorem}{Theorem}

\begin{document}

\title{Spacing and A Large Sieve Type Inequality for Roots of a Cubic Congruence}
\author{Matthew Welsh}
\date{\today}
\maketitle

\begin{abstract}
  Motivated by a desire to understand the distribution of roots of cubic congruences, we re-derive a parametrization of roots $\nu \pmod m$ of $X^3 \equiv 2 \pmod m$ found in \cite{Hooley1978}. Although this parametrization does not lead us here to anything towards proving equidistribution of the sequence $\frac{\nu}{m} \pmod 1$, we are able to prove spacing results, and then a large sieve type inequality, which we view as analogous to the large sieve inequality for roots of quadratic congruences found by Fouvry and Iwaniec \cite{FouvryIwaniec1997} in their proof that there are infinitely many primes of the form $n^2 + p^2$.

  The parametrization produces approximations, which are $\asymp m^{2/3}$-torsion points in $\mathbb{R}^2 / \mathbb{Z}^2$ within $O\left(\frac{1}{m} \right)$ of the point $\left( \frac{\nu}{m}, \frac{\nu^2}{m} \right)$. After a digression to characterize those torsion points having the statistically expected spacing, we prove the spacing property alluded to above: that at most a bounded number of the points $\left(\frac{\nu}{m}, \frac{\nu^2}{m}\right)$ with $m \asymp M$ can lie in any disc with radius $\frac{1}{M}$ in $\mathbb{R}^2 / \mathbb{Z}^2$.
\end{abstract}

\section{Introduction}
\label{sec:1}

The study of the distribution of roots of quadratic congruences provides, in our view, one of the most spectacular applications of the spectral theory of automorphic forms on $SL_2(\mathbb{Z}) \backslash SL_2(\mathbb{R})$ to questions in arithmetic. From its initiation in \cite{Hooley1963}, where Hooley incidentally proved the equidistribution of the sequence
\begin{equation*}
  \left\{ \frac{\nu}{m}\in \mathbb{R} / \mathbb{Z} \ :\ \nu^2 + D \equiv 0 \pmod m \right\}
\end{equation*}
towards obtaining an asymptotic with power savings in the error for the average of the divisor function at values of the quadratic polynomial $n^2 + D$, one already sees the role estimates for sums of Kloosterman sums plays in this study. But even before the entrance of the spectral theory of automorphic forms, one sees a number of other fantastic results that depend in a significant way on estimates for the Weyl sums
\begin{equation*}
  \sum_{m\leq x} \sum_{f(v)\equiv 0 (m)} e\left( \frac{h\nu}{m} \right), \quad h\in \mathbb{Z},
\end{equation*}
with $f$ a fixed quadratic polynomial. For example \cite{Iwaniec1978}, where Iwaniec proves that $n^2 + 1$ is infinitely often a prime or product of two primes.

In \cite{Hooley1963}, Hooley produced the bound $\ll_h x^{3/4}(\log x)^2$ for the Weyl sum above (with $h\neq 0$, of course) using the Weil bound for Kloosterman sums. The introduction of the spectral theory of automorphic forms in \cite{Bykovskii1987} produced (for $D > 0$) the bound $\ll_h x^{2/3} \log x $, or $\ll_h x^{1/2}(\log x)^2$ for a smooth version of the Weyl sum, which is best possible, by relating the smoothed version to a Poincar\'{e} series on $\Gamma \backslash \mathbb{H}$, which was then estimated by its spectral expansion and bounds on the Fourier coefficients of automorphic forms. We also mention \cite{Hejhal1984}, which, by a different but similar method, produces bounds of the same strength for specific examples of $D$ (all negative, interestingly). This strategy matured in \cite{DukeFriedlanderIwaniec1995}, where Duke, Friedlander, and Iwaniec extended the above Weyl sum to one with the condition $m\equiv 0 \pmod d$, which they related to a Poincar\'{e} series on $\Gamma_0(d) \backslash \mathbb{H}$, and produced estimates with enough uniformity in $h$ and $d$ to use a sieve idea, coupled with bilinear forms techniques, to prove equidistribution of the sequence
\begin{equation*}
  \left\{ \frac{\nu}{p} \in \mathbb{R} / \mathbb{Z} \ :\ \nu^2 + D \equiv 0 \pmod p \right\},
\end{equation*}
where $p$ is a prime number and $D > 0$.

In contrast to the previously mentioned results \cite{Hooley1963} and \cite{Iwaniec1978} that did not use the spectral theory of automorphic forms, those of the preceding paragraph did not directly transform to Kloosterman sums, leading to the restriction $D > 0$. It was in \cite{Toth2000} that Toth removed the restriction the $D > 0$ from the equidistribution of roots of quadratic congruences to prime moduli in \cite{DukeFriedlanderIwaniec1995} by doing exactly that.

\bigskip

More recent years have seen significant development of the spectral theory of automorphic forms on $SL_3(\mathbb{Z}) \backslash SL_3(\mathbb{R})$, see, for example, the Kuznetsov-like trace formula \cite{Li2010}, \cite{Buttcane2012}, and \cite{Blomer2013}. And while this spectral theory has seen some great applications, see, in addition, \cite{Blomer2017}, to our knowledge direct applications to arithmetic have been limited. Although there is hope, see the introduction to \cite{Buttcane2012} and the introduction to section 4.1 of \cite{Terras1988}, that the future will see applications to questions of a cubic nature in much the same way that the spectral theory of $SL_2(\mathbb{Z}) \backslash SL_2(\mathbb{R})$ has been so fruitful to questions of a quadratic nature, the discussion above being a prime example.

It was in this spirit that we undertook the study of the distribution of cubic congruences, that is with the hope that we might see the spectral theory of $SL_3(\mathbb{Z}) \backslash SL_3(\mathbb{R})$ become applicable. And despite some positive results from our investigations, those in this paper and some others we defer to a future work, it seems to us that, when it comes to the distribution of roots of cubic congruences, this hope may have been misplaced. We will indicate briefly why we believe this below. 

But before continuing with a summary of our work here, let us remark that the equidistribution of the sequence
\begin{equation*}
  \left\{ \frac{\nu}{m} \in \mathbb{R} / \mathbb{Z} \ :\ f(\nu) \equiv 0 \pmod m \right\},
\end{equation*}
with $f$ any irreducible, integral polynomial was proven by Hooley in \cite{Hooley1964}. The estimate for the Weyl sum obtained there, however, is far too weak for any applications along the lines of those in \cite{Hooley1963} or \cite{Iwaniec1978}, as pointed out in \cite{Hooley1964}. For these applications, we imagine one would need at least a power savings.

We also remark that, for $f(X) = X^3 - 2$, the setting in which we will concern ourselves in what follows, another work of Hooley's, \cite{Hooley1978}, has touched on the Weyl sums
\begin{equation*}
  \sum_{m\leq x} \sum_{\nu^3 \equiv 2 (m)} e \left( \frac{h\nu}{m} \right), \quad h\in \mathbb{Z},
\end{equation*}
which we consider our ultimate goal, even though we do not make any progress towards an estimation here. Indeed, in \cite{Hooley1978}, where Hooley develops his $R^*$ conjecture in the course of investigating the largest prime factor of $n^3 + 2$, a parametrization of $m$ and $\nu$ with $\nu^3 \equiv 2 \pmod m$ is given. We obtain this same parametrization by another means in section \ref{sec:parametrization}, which we include not because we developed it without knowledge of \cite{Hooley1978}, but rather because we feel that our method reveals something more, both by being more closely analogous to a method of parametrizing roots of quadratic congruences that fits beautifully with \cite{Bykovskii1987}, \cite{Hejhal1984}, \cite{DukeFriedlanderIwaniec1995}, and \cite{Toth2000}, and also by suggesting an approach to the Weyl sum above. We do not pursue this approach here, but defer also to a future work.

\bigskip

Moving on to an overview of this work, the main result of section \ref{sec:correspondence} is lemma \ref{lemma:correspondence}, a correspondence between roots of $X^3 \equiv 2 \pmod m$ and ideals $I$ in $\mathbb{Q}(2^{1/3})$ such that $\mathbb{Z}[2^{1/3}] / I \cong \mathbb{Z}/m\mathbb{Z}$ as abelian groups. We have chosen this particular congruence to simplify and concretize calculation; some substantive properties we use are that the number field $\mathbb{Q}(2^{1/3})$ has class number one, and that its ring of integers $\mathbb{Z}[2^{1/3}]$ has a monogenic $\mathbb{Z}$-basis $1$, $2^{1/3}$, $2^{2/3}$. We do not believe these properties are essential, and it is our view that it is clear, however technically daunting, how to proceed without them. Somewhat more essential in our view is that $\mathbb{Q}(2^{1/3})$ is a pure cubic field, which plays a fairly crucial role in picking a fundamental domain in section \ref{sec:fundamentaldomain} for the action of the units, see the comments following (\ref{eq:30}).

Returning to section \ref{sec:correspondence}, we remark that such a correspondence with these ideals $I$ with $\mathbb{Z}[2^{1/3}] / I$ cyclic, which we call primitive, can be predicted by Dedekind's primitive ideal theorem, for example. We also note that primitive ideals can be characterized as those ideals $I$ only divisible by degree one prime ideals $\mathfrak{p}$ with the additional constraint that if $\mathfrak{p} \mid I$, then none of the conjugates of $\mathfrak{p}$ divide $I$. This condition arises in the quadratic case as well, however in this setting the condition becomes much simpler: that the ideal is not divisible by any rational integer.

In section \ref{sec:parametrization} we use the correspondence between roots of $\nu^3 \equiv 2 \pmod m$ and primitive ideals to parametrize the former. The result is lemma \ref{lemma:parametrization}. The parametrization basically follows by writing the ideal $I$ as $(\alpha)$ for some $\alpha \in \mathbb{Z}[2^{1/3}]$. Here we are using crucially that $\mathbb{Z}[2^{1/3}]$ has class number one, however one can imagine how to proceed if this is not the case, although we do not do so here: one might break the ideals into their classes and parametrize each separately as $(\alpha)I_j$, with $\alpha \in I_j^{-1}$ and $I_j$ running over some fixed set of representatives of the classes.

As previously mentioned, this parametrization is more or less that of \cite{Hooley1978}. But with our point of view, one begins to suspect that an analogue of writing Weyl sums for roots of quadratic congruences as Poincar\'{e} series for $\Gamma \backslash \mathbb{H}$ will not hold in the cubic setting. Indeed, matrices having the form of $\gamma$ in (\ref{eq:14}) project to a very small subset of cosets in
\begin{equation*}
  \begin{pmatrix} 1 & 0 & 0 \\ * & 1 & 0 \\ * & * & 1 \end{pmatrix} \backslash SL_3(\mathbb{Z}).
\end{equation*}
This is in contrast to the quadratic setting, where any coset of $\begin{pmatrix} 1 & 0 \\ * & 1 \end{pmatrix} \backslash SL_2(\mathbb{Z})$ will give an analogue to the matrix $\gamma$. There, this is the property that allows one to express a sum over roots of quadratic congruences as a sum over these cosets, and hence provides key step towards writing the Weyl sums as a Poincar\'{e} series.

One can somewhat thicken the kinds of allowable $\gamma$, but only slightly, by removing the requirement that the ideal generated by $\alpha$ is primitive. These ideals will no longer correspond to $\nu^3 \equiv 2 \pmod m$, but rather, it turns out, they will usually correspond to pairs of roots, usually with different moduli. This thickening however, still does not lead to a sum over all the cosets of $SL_3(\mathbb{Z})$, but still a certain a subset of them which one can parametrize by certain ideals in a quadratic field. We will pursue this in future work. 

Despite this apparent setback in our study of the distribution of roots of cubic congruences, we proceed as one would in many of the works cited above. Specifically, we find an approximation to the root $\frac{\nu}{m}$ analogous to the approximation used in \cite{Hooley1963} and \cite{Iwaniec1978}. In our view, this approximation is of crucial importance to the work of Toth \cite{Toth2000}. Indeed, it seems to us that this approximation allows Toth to utilize the Bruhat (double coset) decomposition of $SL_2(\mathbb{Z})$ to transform the Poincar\'{e} series coming from the Weyl sum for quadratic congruences, thus, much like the computation of the Fourier expansion of a standard Poincar\'{e} series, producing Kloosterman sums. 

In section \ref{sec:approximation} we use the parametrization of section \ref{sec:parametrization} to find approximations to the root $\frac{\nu}{m}$. Originally we found the approximation by performing a Bruhat decomposition on the matrix $\gamma$, in the spirit of how one obtains the Fourier expansion of Poincar\'{e} series on $SL_3(\mathbb{Z}) \backslash SL_3(\mathbb{R})$, see \cite{BumpFriedbergGoldfeld1988}. However, in section \ref{sec:approximation} we present an alternative method for producing the approximation, mainly because the method presented provides more insight into the following sections, which deviate significantly from what we have discussed so far.

Even with the obstruction to creating a Poincar\'{e} series for $SL_3(\mathbb{Z}) \backslash SL_3(\mathbb{R})$, we did not find producing an approximation to $\frac{\nu}{m}$ to be a pointless exercise. Inspired by \cite{FouvryIwaniec1997}, where Fouvry and Iwaniec used the approximation to roots of quadratic congruences to derive spacing results, and hence a large sieve inequality, for these roots along the way to proving that there are infinitely many primes of the form $n^2 + p^2$ (see also \cite{FriedlanderIwaniec1998}, where the spacing property is also crucially used]), we can investigate spacing properties of roots of our cubic congruence.

Although much coarser than equidistribution, spacing statistics are not to be overlooked, as they can provide information, albeit less, at much finer scales than one could hope for equidistribution. For an easy illustration, consider the set of all fractions $\frac{a}{q} \in \mathbb{R} / \mathbb{Z}$ with $q\leq Q$. Even using the Riemann hypothesis to bound the corresponding Weyl sum, which here are sums of Ramanujan sums, one can only get equidistribution to scales a bit larger than $\frac{1}{Q}$, for the simple reason that the interval $\left( 0, \frac{1}{Q} \right)$ contains none of these fractions. But still, one can easily show that these fractions are spaced by $\geq \frac{1}{Q^2}$, information at a much smaller scale.

The result of sections \ref{sec:approximation} and \ref{sec:fundamentaldomain}, where, as mentioned above, we make use of the fact that $\mathbb{Q}(2^{1/3})$ is a pure cubic field to pick a fundamental domain for the action of the units, is theorem \ref{theorem:approximation}. The approximations to $\frac{\nu}{m}$ there, while succeeding in the sense that we obtain a good error term, $O\left( \frac{1}{m}\right)$, and a smaller denominator, size $\asymp m^{2/3}$, are weak in the sense that one should typically be able to find a fraction with much smaller denominator, $\ll m^{1/2}$, within the range of our error term. Even more, the numerator and denominator of our approximations depend on each other in a fairly complicated manner, a fact we see as related to the previously discussed thin set of cosets that correspond to roots. Neither of these weaknesses occur for the approximations in the quadratic setting.

Nonetheless, we observe a new phenomenon from the quadratic setting: the method for producing approximations to $\frac{\nu}{m}$ simultaneously produces approximations to $\frac{\nu^2}{m}$ with the same denominator. Taking this new information into account, we see that our method, just as in the quadratic case, produces an optimal result in a certain  sense.

In section \ref{sec:spacing} we first make a small digression to discuss the spacing between points $\left( \frac{r}{q}, \frac{s}{q}\right)$, which we can interpret as $q$-torsion points in $\mathbb{R}^2/\mathbb{Z}^2$, with the goal of using our approximation in theorem \ref{theorem:approximation} to prove spacing results between the points $\left( \frac{\nu}{m}, \frac{\nu^2}{m}\right) \in \mathbb{R}^2 / \mathbb{Z}^2$. The result of this digression is lemma \ref{lemma:spacingtorsionpoints}, which says that the spacing between torsion points is controlled by the sizes of the coefficients of the integral lines passing through the point -- points lying on lines with small coefficients, for example the diagonal $X = Y$, being allowed to be much closer together.

Naturally, we close section \ref{sec:spacing} by proving that the approximations found in theorem \ref{theorem:approximation} do not lie on any lines with small coefficients, thus leading to theorem \ref{theorem:spacingroots}, the cause for half the title of this paper. The theorem states that only a bounded number of the points $\left( \frac{\nu}{m}, \frac{\nu^2}{m} \right)$ with $m \asymp M$ can be inside a disc of radius $\frac{1}{M}$ in $\mathbb{R}^2 / \mathbb{Z}^2$.

We close the paper in section \ref{sec:largesieve} by using theorem \ref{theorem:spacingroots} to derive a large sieve type inequality in the spirit of \cite{FouvryIwaniec1997}. Despite the transformations used in going from spacing to a large sieve inequality being completely standard, we record them here. We remark that although the resulting inequality in theorem \ref{theorem:largesieve} is optimal in certain ranges, it gives nothing nontrivial in others -- the novelty of our inequality is the range in which it is optimal.

{\bf Acknowledgments}. I would like the thank Steve Miller for his encouragement at the very beginning of this research, and my advisor, Henryk Iwaniec, for his advice and guidance throughout. I would also like to Ram Murty, who, after attending a presentation of this work at the {\em Strength in Numbers} conference, introduced me to \cite{Hooley1978}, which contains the results of the first few sections here. 

\section{Correspondence between Roots and Primitive Ideals}
\label{sec:correspondence}

In this section we establish an explicit correspondence between certain ideals in $\mathcal{O} = \mathbb{Z}[2^{1/3}]$ and roots of the cubic congruence $x^3 \equiv 2 \pmod m$. 

Let $I$ be an ideal in $\mathcal{O}$. Fixing the $\mathbb{Z}$-basis $\{1, 2^{1/3}, 2^{2/3}\}$ of $\mathcal{O}$, we pick the unique $\mathbb{Z}$-basis $\{\omega_1, \omega_2, \omega_3\}$ of $I$ so that
\begin{equation}
  \label{eq:1}
  \begin{pmatrix} \omega_1 \\ \omega_2 \\ \omega_3 \end{pmatrix} = A \begin{pmatrix} 1 \\ 2^{1/3} \\ 2^{2/3} \end{pmatrix},
\end{equation}
with $A$ an integer matrix in Hermite normal form, which is to say that
\begin{equation}
  \label{eq:2}
  A = \begin{pmatrix} a_{11} & 0 & 0 \\ a_{21} & a_{22} & 0 \\ a_{31} & a_{32} & a_{33} \end{pmatrix}
\end{equation}
with $a_{11}, a_{22}, a_{33} > 0$ and $0\leq a_{21}, a_{31} < a_{11}$, $0\leq a_{32} < a_{22}$.

Since $2^{1/3}$ generates $\mathcal{O}$, $I$ being an ideal is equivalent to $2^{1/3}I$ being a sublattice of $I$. In other words, we need $2^{1/3}$ to act by an integral matrix with respect to the basis (\ref{eq:1}). And since
\begin{equation}
  \label{eq:3}
  2^{1/3} \begin{pmatrix} 1 \\ 2^{1/3} \\ 2^{2/3} \end{pmatrix} = \begin{pmatrix} 0 & 1 & 0 \\ 0 & 0 & 1 \\ 2 & 0 & 0 \end{pmatrix} \begin{pmatrix} 1 \\ 2^{1/3} \\ 2^{2/3} \end{pmatrix},
\end{equation}
we are asking for
\begin{equation}
  \label{eq:4}
  A \begin{pmatrix} 0 & 1 & 0 \\ 0 & 0 & 1 \\ 2 & 0 & 0 \end{pmatrix} A^{-1} = \begin{pmatrix} -\frac{a_{21}}{a_{22}} & \frac{a_{11}}{a_{22}} & 0 \\ -\frac{a_{21}^2}{a_{11}a_{22}} + \frac{a_{21}a_{32}}{a_{11}a_{33}} -\frac{a_{22}a_{31}}{a_{11}a_{33}} & \frac{a_{21}}{a_{22}} - \frac{a_{32}}{a_{33}} & \frac{a_{22}}{a_{33}} \\ 2\frac{a_{33}}{a_{11}} - \frac{a_{21}a_{31}}{a_{11}a_{22}} + \frac{a_{21}a_{32}^2}{a_{11}a_{22}a_{33}} - \frac{a_{31}a_{32}}{a_{11}a_{33}} & \frac{a_{31}}{a_{22}} - \frac{a_{32}^2}{a_{22}a_{33}} & \frac{a_{32}}{a_{33}} \end{pmatrix}  
\end{equation}
to be an integer matrix.

From the $(1,2)$ and $(2,1)$ entries, we see that $a_{33} \mid a_{22}$ and $a_{22} \mid a_{11}$. Moreover, from the $(1,1)$ and $(3,3)$ entries, $a_{22}\mid a_{21}$ and $a_{33} \mid a_{32}$. And finally, from the $(3,2)$ entry, $a_{33}\mid a_{31}$. These facts together show that $A$ can be put into Smith normal form by multiplying on the right by a lower triangular matrix in $SL_3(\mathbb{Z})$. Whence $a_{33}$, $a_{22}$, and $a_{11}$ are the invariant factors of $\mathcal{O}/I$. We remark that this fact, that the diagonal entries in Hermite normal form are the invariant factors, does not hold for general lattices, and depend here on the ideal structure. 

At this point it is natural to make the assumption that $\mathcal{O}/I$ is cyclic (we call $I$ primitive in this case), so $a_{33} = a_{22} = 1$ and $a_{11} = N(I) = m$, say. Now $a_{32} = 0$ and we have
\begin{equation}
  \label{eq:5}
  A \begin{pmatrix} 0 & 1 & 0 \\ 0 & 0 & 1 \\ 2 & 0 & 0 \end{pmatrix} A^{-1} = \begin{pmatrix} - a_{21} & m & 0 \\ -\frac{a_{21}^2 + a_{31}}{m} & a_{21} & 1 \\ \frac{2 - a_{21}a_{31}}{m} & a_{31} & 0 \end{pmatrix}. 
\end{equation}
For this matrix to be integral, we need
\begin{equation}
  \label{eq:6}
  \begin{split}
    a_{31} & \equiv -a_{21}^2 \pmod m, \\
    a_{21}a_{31} & \equiv 2 \pmod m. \\
  \end{split}
\end{equation}

Substituting the first congruence into the second gives $a_{21} \equiv - \nu \pmod m$ with $\nu$ a root of $X^3 \equiv 2 \pmod m$, and then $a_{31} \equiv -\nu^2$.

We have obtained
\begin{lemma}
  \label{lemma:correspondence}
  Let $I$ be a primitive ideal of $\mathcal{O}$. Then $I$ has a unique basis of the form
  \begin{equation}
    \label{eq:7}
    \begin{pmatrix} \omega_1 \\ \omega_2 \\ \omega_3 \end{pmatrix} = \begin{pmatrix} m & 0 & 0 \\ -\nu & 1 & 0 \\ -\nu^2 & 0 & 1 \end{pmatrix} \begin{pmatrix} 1 \\ 2^{1/3} \\ 2^{2/3} \end{pmatrix} 
  \end{equation}
  where $\nu$ is a root of the congruence $X^3 \equiv 2 \pmod m$. Here, for the uniqueness to hold, $\nu$ and $\nu^2$ are considered as residue classes $\pmod m$.
  
  Conversely, given $m$ and $\nu \pmod m$ such that $\nu^3 \equiv 2 \pmod m$, the lattice with basis given by (\ref{eq:7}) is a primitive ideal of $\mathcal{O}$. 
\end{lemma}

\section{Parametrization of Roots}
\label{sec:parametrization}

Having established this explicit correspondence between primitive ideals $I$ of norm $m$ and roots of $\nu^3 \equiv 2 \pmod m$, we can now obtain a parametrization of such pairs $\nu$, $m$ by writing down different bases for the ideals $I$. The key observation here is that $\mathcal{O} = \mathbb{Z}[2^{1/3}]$ has class number $1$, so
\begin{equation}
  \label{eq:8}
  I = \left( a + b2^{1/3} + c 2^{2/3} \right) = (\alpha)
\end{equation}
for some integers $a$, $b$, and $c$, unique up to the action of the group of units
\begin{equation}
  \label{eq:9}
  \mathcal{U} = \langle -1, 1 + 2^{1/3} + 2^{2/3} \rangle.
\end{equation}

It follows that a natural $\mathbb{Z}$-basis of $I$ is
\begin{equation}
  \label{eq:10}
  \begin{pmatrix} \alpha \\ 2^{1/3} \alpha \\ 2^{2/3} \alpha \end{pmatrix} = \begin{pmatrix} a & b & c \\ 2c & a & b \\ 2b & 2c & a \end{pmatrix} \begin{pmatrix} 1 \\ 2^{1/3} \\ 2^{2/3} \end{pmatrix}. 
\end{equation}
Comparing this basis with the one in (\ref{eq:7}), we see that there must be a matrix $\gamma \in SL_3(\mathbb{Z})$ such that
\begin{equation}
  \label{eq:11}
  \gamma \begin{pmatrix} m & 0 & 0 \\ -\nu & 1 & 0 \\ -\nu^2 & 0 & 1 \end{pmatrix} = \begin{pmatrix} a & b & c \\ 2c & a & b \\ 2b & 2c & a \end{pmatrix}.
\end{equation}
Note that, since $m > 0$, in order to have $\det \gamma = 1$, we are requiring that
\begin{equation}
  \label{eq:12}
  \det \begin{pmatrix} a & b & c \\ 2c & a & b \\ 2b & 2c & a \end{pmatrix} = a^3 + 2b^3 + 4c^3 - 6abc > 0,
\end{equation}
which we may assume by replacing $\alpha$ with $-\alpha$ if necessary (later we will in fact choose $\alpha$ to be in a specific fundamental domain of the action of the units, where it is easily verified that (\ref{eq:12}) holds, see section \ref{sec:fundamentaldomain}). 

We note that since
\begin{equation}
  \label{eq:13}
  \gamma \begin{pmatrix} 0 & 0 \\ 1 & 0 \\ 0 & 1 \end{pmatrix} = \begin{pmatrix} b & c \\ a & b \\ 2c & a \end{pmatrix},
\end{equation}
we have
\begin{equation}
  \label{eq:14}
  \gamma = \begin{pmatrix} u & b & c \\ v & a & b \\ w & 2c & a \end{pmatrix} 
\end{equation}
for some integers $u$, $v$, and $w$, which are determined up to multiplication on the right by matrices of the form
\begin{equation}
  \label{eq:15}
  \begin{pmatrix} 1 & 0 & 0 \\ * & 1 & 0 \\ * & 0 & 1 \end{pmatrix}.
\end{equation}

For the matrix (\ref{eq:14}) to have determinant $1$, we must have
\begin{equation}
  \label{eq:16}
  u(a^2 - 2bc) + v(2c^2 - ab) + w(b^2 - ac) = 1.
\end{equation}
In particular
\begin{equation}
  \label{eq:17}
  gcd\left( a^2 - 2bc, 2c^2 - ab, b^2 - ac \right) = 1,
\end{equation}
which gives a criterion for the ideal $I$ to be primitive in terms of its generator.

Now
\begin{equation}
  \label{eq:18}
  \gamma^{-1} = \begin{pmatrix} a^2 - 2bc & 2c^2 - ab & b^2 - ac \\ bw - av & au - cw & cv - bu \\ 2cv - aw & bw - 2cu & au - bv \end{pmatrix},
\end{equation}
so multiplying the matrix (\ref{eq:10}) by $\gamma^{-1}$ and equating it to the matrix (\ref{eq:7}), we obtain
\begin{equation}
  \label{eq:19}
  \begin{split}
    m & = a^3 + 2b^3 + 4c^3 - 6abc \\
    -\nu & \equiv a(bw - av) + 2c(au-cw) + 2b(cv - bu) \pmod m.
  \end{split}
\end{equation}
Note that the first line of (\ref{eq:19}) is no surprise: the right side is just the norm-form for $\mathcal{O}$. Also note that the choice of $u$, $v$, and $w$ corresponds to multiplying the matrix (\ref{eq:18}) on the left by matrices of the form (\ref{eq:15}). In other words, the choice of $u$, $v$, and $w$, corresponds to the choice of representatives for $\nu$ and $\nu^2 \pmod m$. 

We summarize what we have obtained as follows
\begin{lemma}
  \label{lemma:parametrization}
  Let $a$, $b$, $c$ be any integers satisfying (\ref{eq:17}) and (\ref{eq:12}). Let $u$, $v$, and $w$ be integers satisfying
  \begin{equation}
    \label{eq:20}
    u(a^2 - 2bc) + v(2c^2 - ab) + w(b^2 - ac) = 1.
  \end{equation}
  Then
  \begin{equation}
    \label{eq:21}
    \begin{split}
      m & = a^3 + 2b^3 + 4c^3 - 6abc \\
      -\nu & = a(bw - av) + 2c(au-cw) + 2b(cv - bu),
    \end{split}
  \end{equation}
  satisfy $\nu^3 \equiv 2 \pmod m$.

  Moreover, $a$, $b$, and $c$ for which the $a + b 2^{1/3} + c 2^{2/3}$ lie in different orbits of the action of $\mathcal{U}$ on $\mathcal{O}$ give distinct pairs $m$ and $\nu \pmod m$. 
\end{lemma}

\section{Approximating the Roots}
\label{sec:approximation}

To obtain an approximation to $\frac{\nu}{m}$, we return to the equation (\ref{eq:11}) with $\gamma$ as in (\ref{eq:14}). We have
\begin{equation}
  \label{eq:22}
  \begin{pmatrix} u & b & c \\ v & a & b \\ w & 2c & a \end{pmatrix} \begin{pmatrix} m & 0 & 0 \\ -\nu & 1 & 0 \\ -\nu^2 & 0 & 1 \end{pmatrix} = \begin{pmatrix} a & b & c \\ 2c & a & b \\ 2b & 2c & a \end{pmatrix}.
\end{equation}

Examining this equality for the $(1,1)$ entry on the right-hand side, we obtain
\begin{equation}
  \label{eq:23}
  u - b \frac{\nu}{m} - c \frac{\nu^2}{m} = \frac{a}{m}
\end{equation}
upon dividing by $m$. We expect, and will ensure later by picking $\alpha = a + b2^{1/3} + c2^{2/3}$ in a specific fundamental domain for the action of the units, that $a \ll m^{1/3}$. Accordingly, we expect the right hand side of (\ref{eq:23}) to be small, specifically $\ll m^{-2/3}$. We can interpret this geometrically as the point $\left( \frac{\nu}{m}, \frac{\nu^2}{m}\right)$ lying close to the line $bX + cY = u$.

Similarly, by inspecting the $(2,1)$ and $(3,1)$ entries of the right side of (\ref{eq:22}), we expect that $\left(\frac{\nu}{m}, \frac{\nu^2}{m}\right)$ will also lie close to the lines $aX + bY = v$ and $2cX + aY = w$. Now, if the triangle with sides these three lines is not too obtuse, that is if the angles between pairs of these lines are neither too large nor small, then the point $\left( \frac{\nu}{m}, \frac{\nu^2}{m} \right)$ will be close to the intersection of each pair of lines, namely
\begin{equation}
  \label{eq:24}
  \left(\frac{bu-cv}{b^2 - ac}, \frac{bv - au}{b^2 -ac}\right), \; \left( \frac{cv - au}{2c^2 - ab}, \frac{2cu - bv}{2c^2 - ab}\right), \; \left(\frac{au - bv}{a^2 - 2bc}, \frac{av - 2cu}{a^2 - 2bc}\right).
\end{equation}

We can see this explicitly by solving for $\frac{\nu}{m}$ and $\frac{\nu^2}{m}$ in each of the pairs of the three equations coming from (\ref{eq:22}) as above. For example, we have
\begin{equation}
  \label{eq:25}
  \frac{1}{m} \begin{pmatrix} b & c \\ a & b \end{pmatrix} \begin{pmatrix} \nu \\ \nu^2 \end{pmatrix} = \begin{pmatrix} u \\ v \end{pmatrix} - \frac{1}{m}\begin{pmatrix} a \\ 2c \end{pmatrix},
\end{equation}
so
\begin{equation}
  \label{eq:26}
  \frac{1}{m} \begin{pmatrix} \nu \\ \nu^2 \end{pmatrix} = \frac{1}{b^2 - ac} \begin{pmatrix} bu - cv \\ bv - au \end{pmatrix} + \frac{1}{m(b^2 - ac)}\begin{pmatrix} 2c^2 - ab \\ a^2 - 2bc \end{pmatrix}.
\end{equation}
If we have that all of the quantities $b^2 -ac$, $2c^2 - ab$, and $2c^2 - ab$ are of the same size, which will be $\asymp m^{2/3}$, then we see that $\left(\frac{\nu}{m}, \frac{\nu^2}{m}\right)$ is indeed within $\ll \frac{1}{m}$ of the first intersection point listed in (\ref{eq:24}). We remark that, as the determinants of the matrices
\begin{equation}
  \label{eq:27}
  \begin{pmatrix} b & c \\ a & b \end{pmatrix}, \; \begin{pmatrix} 2c & a \\ b & c \end{pmatrix}, \; \begin{pmatrix} a & b \\ 2c & a \end{pmatrix},
\end{equation}
control on these quantities does indeed correspond to control on the size of the angles between the three pairs of lines mentioned above. 

Carrying this out for the other two pairs of equations shows that the other points in (\ref{eq:24}) approximate $\left(\frac{\nu}{m}, \frac{\nu^2}{m}\right)$ to $\ll \frac{1}{m}$, as long as $b^2 -ac$, $2c^2 - ab$, and $a^2 - 2bc$ are all of the same order of magnitude, $m^{2/3}$. In the following section we ensure this hypothesis by picking $\alpha = a + b2^{1/3} + c2^{2/3}$ to be in an appropriate fundamental domain for the action of the units, and also verify that for our choice $a$, $b$, and $c$ will all be $\ll m^{1/3}$, although this hypothesis turned out to be inessential to our work so far.


\section{Picking the Fundamental Domain}
\label{sec:fundamentaldomain}

Towards picking a fundamental domain for which $b^2 -ac$, $2c^2-ab$, and $a^2 - 2bc$ are all $\asymp m^{2/3}$, we begin by observing that
\begin{equation}
  \label{eq:28}
  (a^2 - 2bc) + (2c^2 - ab)2^{1/3} + (b^2 - ac)2^{2/3} = \frac{m}{\alpha},
\end{equation}
where, as before, $\alpha = a + b2^{1/3} + c2^{2/3}$. This can of course be verified directly by multiplying both sides by $\alpha$, but we can also see it also by work we have already done. Indeed, we've already seen, (\ref{eq:13}), that $\gamma$ shares the second and third columns of the matrix by which $\alpha$ acts on the basis $1$, $2^{1/3}$, and $2^{2/3}$, see (\ref{eq:10}). Moreover, the quantities $b^2 - ac$, $2c^2 - ab$, and $a^2 - 2bc$ came naturally in the previous section as the minors of $\gamma$ in these last columns. As such, they, after dividing by $m$, the determinant of the matrix (\ref{eq:10}), form the first row of the matrix by which $\alpha^{-1}$ acts of the basis $1$, $2^{1/3}$, $2^{2/3}$. From this, (\ref{eq:28}) follows.

Now, making use of (\ref{eq:28}), we can write
\begin{equation}
  \label{eq:29}
  \begin{split}
    a^2 - 2bc & = \frac{m}{3} \mathrm{Tr} (\alpha^{-1}), \\
    2c^2 - ab & = \frac{m}{6} \mathrm{Tr} (2^{2/3}\alpha^{-1}), \\
    b^2 - ac & = \frac{m}{6} \mathrm{Tr} (2^{1/3}\alpha^{-1}).
  \end{split}
\end{equation}
With this in mind, we will construct a fundamental domain in which we can control these traces.

For $\beta \in K = \mathbb{Q}(2^{1/3})$, let $\beta^{(1)}$ be the real embedding and $\beta^{(j)}$, $j=2,3$, be the complex embeddings. And for $C > 0$ a constant to be determined, set
\begin{equation}
  \label{eq:30}
  \mathcal{D}_1 = \left\{ \beta \in K\ :\ C|N(\beta)|^{1/3} < \beta^{(1)} \leq C\varepsilon^{(1)} |N(\beta)|^{1/3} \right\},
\end{equation}
where $\varepsilon = 1 + 2^{1/3} + 2^{2/3}$ is the fundamental unit. $\mathcal{D}_1$ is clearly a fundamental domain for the action of the units on $K$, and also, perhaps a bit more surprisingly, that if $\beta\in \mathcal{D}_1$ then $2^{1/3}\beta$ is as well.

For $\beta\in K$, we have
\begin{equation}
  \label{eq:31}
  |\beta^{(2)}|^2 = |\beta^{(2)}\beta^{(3)}| = \frac{|N(\beta)|}{|\beta^{(1)}|},
\end{equation}
so for $\beta \in \mathcal{D}_1$, we have
\begin{equation}
  \label{eq:32}
  |\beta^{(2)}| \leq C^{-1/2} N(\beta)^{1/3}.
\end{equation}
Hence
\begin{equation}
  \label{eq:33}
  |\mathrm{Tr}(\beta))| = \left| \beta^{(1)} + \beta^{(2)} + \beta^{(3)} \right| \leq (\varepsilon^{(1)} C + 2C^{-1/2}) N(\beta)^{1/3},
\end{equation}
and
\begin{equation}
  \label{eq:34}
  |\mathrm{Tr}(\beta)| \geq \beta^{(1)} - 2\left|\beta^{(2)}\right| \geq (C - 2C^{-1/2}) N(\beta)^{1/3}.
\end{equation}
So picking $C = 2$, say, gives
\begin{equation}
  \label{eq:35}
  |Tr(\beta)| \asymp N(\beta)^{1/3}
\end{equation}
for $\beta \in \mathcal{D}_1$.

Picking the fundamental domain for $\alpha$ to be
\begin{equation}
  \label{eq:36}
  \mathcal{D} = \left\{ \beta \in K \ :\ \beta^{-1} \in \mathcal{D}_1 \right\},
\end{equation}
we have by (\ref{eq:35}) and (\ref{eq:29}), $a^2 - 2bc \asymp m^{2/3}$. Moreover, applying the comment following (\ref{eq:30}), we can also conclude that $2c^2 - ab$ and $b^2 - ac$ are $\asymp m^{2/3}$ as well.

Together with the previous section, we have proved
\begin{theorem}
  \label{theorem:approximation}
  Let $\nu \pmod m$ be a root of $X^3 \equiv 2 \pmod m$. From lemma 1, let $(\alpha)$ be the corresponding primitive ideal in $\mathbb{Z}[2^{1/3}]$ with $\alpha = a + b2^{1/3} + c2^{2/3}$ in the fundamental domain $\mathcal{D}$ as defined in (\ref{eq:36}). And let $u$, $v$, and $w$ be integers satisfying (\ref{eq:16}). Then the points
  \begin{equation}
    \label{eq:37}
    \left(\frac{bu-cv}{b^2 - ac}, \frac{bv - au}{b^2 -ac}\right), \; \left( \frac{cv - au}{2c^2 - ab}, \frac{2cu - bv}{2c^2 - ab}\right), \; \left(\frac{au - bv}{a^2 - 2bc}, \frac{av - 2cu}{a^2 - 2bc}\right),
  \end{equation}
are all within $\ll \frac{1}{m}$ of the point $\left(\frac{\nu}{m}, \frac{\nu^2}{m}\right)$.
\end{theorem}

We close this section by showing that for $\alpha \in \mathcal{D}$, $a$, $b$, and $c$ are all $\ll m^{1/3}$. This turned out not to be necessary for theorem 1, but we will use it in the following section when we investigate the spacing between the points (\ref{eq:37}). There are a number of ways to see this bound, the one we show below is very quick using what we've already done.

Writing
\begin{equation}
  \label{eq:38}
  \alpha = \frac{1}{m} \frac{N\left(\frac{m}{\alpha}\right)}{\frac{m}{\alpha}},
\end{equation}
we can apply (\ref{eq:28}) with $\frac{m}{\alpha}$ taking the role of $\alpha$, which means that $a^2 - 2bc$, $2c^2 - ab$, and $b^2 - ac$ take the roles of $a$, $b$, and $c$ respectively. We obtain
\begin{equation}
  \label{eq:39}
  \begin{split}
    a & = \frac{1}{m} \left( (a^2 - 2bc)^2 - 2(2c^2 - bc) (b^2 - ac) \right) \\
    b & = \frac{1}{m} \left( 2(b^2 - ac)^2 - (a^2 - 2bc)(2c^2 - ab) \right) \\
    c & = \frac{1}{m} \left( (2c^2 - ab)^2 - (a^2 - 2bc)(b^2 - ac) \right), \\
  \end{split}
\end{equation}
which can also be verified directly. Using just that $b^2 - ac$, $2c^2 - ab$, and $a^2 - 2bc$ are all $\ll m^{2/3}$, we can immediately conclude from (\ref{eq:32}) that $a$, $b$, and $c$ are all $\ll m^{1/3}$.

\section{Spacing of Torsion Points on $\mathbb{R}^2 / \mathbb{Z}^2$}
\label{sec:spacing}

As discussed in the introduction, the approximations to $\frac{\nu}{m}$ given by theorem 1, for example
\begin{equation}
  \label{eq:40}
  \frac{\nu}{m} = \frac{bu-cv}{b^2 - ac} + O\left( \frac{1}{m}\right),
\end{equation}
are not optimal in the sense of Dirichlet's theorem on Diophantine approximation. Typically we should be able to find fractions with denominator of size $m^{1/2}$ in an interval of length $\frac{1}{m}$ around $\frac{\nu}{m}$, but the approximation (\ref{eq:40}) only gives the much larger denominator, size $m^{2/3}$. This is in contrast to the analogous approximations to the roots of quadratic congruences outlined in the introduction, where this kind of optimality lead to very strong spacing results that in turn give optimal large sieve inequalities. Moreover, in the quadratic setting, one can easily and uniquely recover the root of the congruence from the approximation. But in our cubic setting, while it seems possible based on some numerical evidence that there is a one-to-one correspondence between the roots and the approximations, we have unfortunately been unable to prove anything close to this.

However, if we consider the approximation to $\frac{\nu}{m}$ and $\frac{\nu^2}{m}$ simultaneously, as we have done in theorem 1, we do recover the kind of optimality suggested by Dirichlet's Diophantine approximation theorem (now simultaneous in two variables). But in contrast to the one-variable approximation, it is not immediate that we can conclude the kind of spacing between roots required for a strong large sieve type inequality. Indeed, it is easy to construct examples of torsion points $\left( \frac{r}{q}, \frac{s}{q} \right)$, $r$, $s$, $q$ coprime (not necessarily pairwise coprime) integers with $q > 0$, on $\mathbb{R}^2/\mathbb{Z}^2$ with torsion $q \leq Q$ that are much closer than the statistically expected $\frac{1}{Q^{3/2}}$. For example, points on the diagonal: $\left(\frac{r}{q}, \frac{r}{q}\right)$ and $\left( \frac{r_1}{q_1}, \frac{r_1}{q_1}\right)$ can be as close as $\frac{\sqrt{2}}{qq_1}$ in the Euclidean metric.

Our goal, then, for this section is to show that the approximations, which are of the form $\left( \frac{r}{q}, \frac{s}{q}\right)$, have the typical spacing from any torsion points with torsion $\ll m^{2/3}$. And towards this end, we develop a characterization of such torsion points.

Let $\left( \frac{r}{q}, \frac{s}{q}\right)$ and $\left( \frac{r_1}{q_1}, \frac{s_1}{q_1}\right)$ be representatives of distinct torsion points in $\mathbb{R}^2/\mathbb{Z}^2$. Let $AX + BY = C$ with $A$, $B$, $C$ integers such that $gcd(A,B,C) = 1$ be the equation of the line between the two. The coprimality condition on the coefficients implies that
\begin{equation}
  \label{eq:41}
  \mathbb{Z}^3 \cap \mathrm{Null} \begin{pmatrix} q & q_1 \\ -r & -r_1 \\ -s & -s_1 \end{pmatrix} = \mathbb{Z}\begin{pmatrix} C & A & B \end{pmatrix}.
\end{equation}

On the other hand the cross product of $\begin{pmatrix} q & -r & -s \end{pmatrix}$ and $\begin{pmatrix} q_1 & - r_1 & -s_1 \end{pmatrix}$ is in this null space, so we can conclude that
\begin{equation}
  \label{eq:42}
  \begin{pmatrix} rs_1 - r_1 s & q s_1 - q_1 s & rq_1 - r_1 q \end{pmatrix} = k \begin{pmatrix} C & A & B \end{pmatrix},
\end{equation}
for some integer $k$. Since the torsion points are distinct, we know that $k\neq 0$, so $|k|\geq 1$ in fact. We have
\begin{equation}
  \label{eq:43}
  \left| \frac{r}{q} - \frac{r_1}{q_1} \right| = \frac{|rq_1 - r_1 q|}{qq_1} \geq \frac{|B|}{qq_1},
\end{equation}
and similarly
\begin{equation}
  \label{eq:44}
  \left| \frac{s}{q} - \frac{s_1}{q_1}\right| \geq \frac{|A|}{qq_1}.
\end{equation}

From (\ref{eq:43}) and (\ref{eq:44}) we see that the size of $|A|$ and $|B|$ from lines $AX + BY = C$ passing through a representative of a torsion point control the spacing from this representative to a representative of any another torsion point. Moreover, we observe that if $AX + BY = C$ passes through a representative $\left( \frac{r}{q}, \frac{s}{q}\right)$, then $AX + BY = C + kA + lB$ passes through another representative $\left( \frac{r}{q} + k, \frac{s}{q} + l\right)$, where $k$ and $l$ are integers. Hence the set of $\begin{pmatrix} A & B \end{pmatrix}$ under consideration will not depend on the choice of representative.

We have
\begin{lemma}
  \label{lemma:spacingtorsionpoints}
  Let $\left(\frac{r}{q}, \frac{s}{q}\right)$ be a torsion point in $\mathbb{R}^2 / \mathbb{Z}^2$. Then the distance, measured by the projection of the Euclidean metric, from $\left(\frac{r}{q}, \frac{s}{q}\right)$ to any different torsion point with torsion $\leq Q$ will be at least
  \begin{equation}
    \label{eq:45}
    \frac{1}{qQ} \min \left\{ \sqrt{A^2 + B^2} \ :\ A\frac{r}{q} + B\frac{s}{q} \in \mathbb{Z} \right\}.
  \end{equation}

\end{lemma}

Before applying this lemma to our approximations to roots of $x^3 \equiv 2 \pmod m$, we remark that the set of all $\begin{pmatrix} A & B \end{pmatrix}$ such that some fixed representative $\left( \frac{r}{q}, \frac{s}{q}\right)$ lies on a line $AX + BY = C$ forms a sublattice of $\mathbb{Z}^2$. As mentioned previously, this lattice is independent of the representative $\left( \frac{r}{q}, \frac{s}{q} \right)$ chosen for the torsion point $\mathbb{R}^2 / \mathbb{Z}^2$. And if this lattice, properly oriented and normalized to have co volume $1$, does not lie too high in the cusp of $SL_2(\mathbb{Z}) \backslash SL_2(\mathbb{R})$, then the shortest vector in the lattice will have norm about the square root of the covolume.

Let's consider the point $\left( \frac{bu - cv}{b^2 - ac}, \frac{bv - au}{b^2 - ac} \right)$, the first approximation to $\left( \frac{\nu}{m}, \frac{\nu^2}{m}\right)$ listed in (\ref{eq:37}) of theorem 1. From the way it was constructed, as the intersection of the lines $bX + cY = u$ and $aX + bY = v$, we can see that the lattice discussed in the previous paragraph is
\begin{equation}
  \label{eq:46}
  \mathrm{Span}_\mathbb{Z} \left\{ \begin{pmatrix} b & c \end{pmatrix}, \begin{pmatrix} a & b \end{pmatrix} \right\}.
\end{equation}
To get this we use that the row vectors $\begin{pmatrix} u & b & c \end{pmatrix}$ and $\begin{pmatrix} v & a & b \end{pmatrix}$ can be completed by a third vector to make a matrix in $SL_3(\mathbb{Z})$, namely $\gamma$, to see that we are not missing any integral lines through the point $\left( \frac{bu - cv}{b^2 - ac}, \frac{bv - au}{b^2 - ac} \right)$. The covolume of this lattice is $b^2 - ac$, which we forced in the previous section to be $\asymp m^{2/3}$. Recall that in that section we also verified that $a$, $b$, and $c$ are all $\ll m^{1/3}$. And from this, we also have
\begin{equation}
  \label{eq:47}
  b^2 - ac \leq \sqrt{(b^2 + c^2)(a^2 + b^2)} \ll m^{1/3} \sqrt{b^2 + c^2},
\end{equation}
so
\begin{equation}
  \label{eq:48}
  \sqrt{b^2 + c^2} \asymp m^{1/3},
\end{equation}
and similarly
\begin{equation}
  \label{eq:49}
  \sqrt{a^2 + b^2} \asymp m^{1/3}.
\end{equation}

Suppose we scale and rotate this lattice so that the vector $\begin{pmatrix} b & c \end{pmatrix}$ becomes $\begin{pmatrix} 1 & 0 \end{pmatrix}$, thereby identifying the lattice with a point in the upper halfplane $\mathbb{H}$, the image of $\begin{pmatrix} a & b \end{pmatrix}$ under this scaling and rotation. After this transformation, the covolume of the lattice is $\gg 1$, whence the point in $\mathbb{H}$ has height$\gg 1$ above the $x$-axis. Moreover, since $\sqrt{a^2 + b^2} \asymp \sqrt{b^2 + c^2}$, the point also has distance $\ll 1$ from the origin. As such, the point lies in a fixed, compact region of $\mathbb{H}$, whence, even after quotienting out by the action of $SL_2(\mathbb{Z})$ on the basis, the lattice lies in a fixed region, bounded away from the cusp.

In accordance with the remarks above, we know that the shortest vector in the lattice will have norm $\asymp$ square root of the covolume, so here the shortest vector will be $\asymp m^{1/3}$. Combining this with lemma 3, and applying the same reasoning to all three approximations to $\left( \frac{\nu}{m}, \frac{\nu^2}{m} \right)$ listed in (\ref{eq:37}) of theorem 1, we have
\begin{lemma}
  \label{lemma:spacingapproximations}
  The approximations to $\left(\frac{\nu}{m}, \frac{\nu^2}{m}\right)$ found in theorem 1, namely
  \begin{equation}
    \label{eq:50}
    \left(\frac{bu-cv}{b^2 - ac}, \frac{bv - au}{b^2 -ac}\right), \; \left( \frac{cv - au}{2c^2 - ab}, \frac{2cu - bv}{2c^2 - ab}\right), \; \left(\frac{au - bv}{a^2 - 2bc}, \frac{av - 2cu}{a^2 - 2bc}\right),
  \end{equation}
  are all spaced by $\gg \frac{1}{m^{1/3}Q}$ from any other torsion point $\left(\frac{r}{q}, \frac{s}{q}\right)$ in $\mathbb{R}^2/ \mathbb{Z}^2$ with torsion $q \leq Q$.  
\end{lemma}

Having this lemma, we can prove spacing results for the set of points
\begin{equation}
  \label{eq:51}
  S = \left\{ \left( \frac{\nu}{m}, \frac{\nu^2}{m} \right) \ :\ \nu^3 \equiv 2 \pmod m,\ M < m \leq 2M \right\}.
\end{equation}
Specifically, we have the following theorem, whose proof is below,
\begin{theorem}
  \label{theorem:spacingroots}
  For any disc $D$ in $\mathbb{R}^2 / \mathbb{Z}^2$ with radius $\frac{1}{M}$,
  \begin{equation}
    \label{eq:52}
    \# (S \cap D)  \ll 1,
  \end{equation}
  where $S$ is as in (\ref{eq:51}). 
\end{theorem}

\begin{proof}
  First we show that we can recover the point $\left( \frac{\nu}{m}, \frac{\nu^2}{m}\right)$ from the three approximations of theorem 1. We have
  \begin{equation}
    \label{eq:53}
    \begin{pmatrix} a^2 - 2bc & 2c^2 - ab & b^2 - ac \\ bw - av & au - cw & cv - bu \\ 2cv - aw & bw - 2cu & au - bv \end{pmatrix} = \begin{pmatrix} u & b & c \\ v & a & b \\ w & 2c & a \end{pmatrix}^{-1},
  \end{equation}
  from which we can observe two facts. For one, the numbers $b^2 - ac$, $cv - bu$, and $au - bv$, as a column of a matrix in $SL_3(\mathbb{Z})$, are coprime, so we can recover these numbers from the fractions $\frac{bu - cv}{b^2 - ac}$ and $\frac{bv - au}{b^2 - ac}$. And two, applying the first observation to the fractions in the other approximations to recover the other columns, we can get $a$, $b$, and $c$ by inverting the matrix in (\ref{eq:53}). From $a$, $b$, and $c$, we can recover $m$ and $\nu$ from the ideal generated by $a + b2^{1/3} + c 2^{2/3}$, or, what is probably more straightforward, we can use the matrix just obtained to get $\nu$ and $m$ directly by using (\ref{eq:11}) or lemma 2.

  By theorem 1, we know that each point in $S \cap D$ will have all three approximations in a potentially larger disc, but still with radius $\ll \frac{1}{M}$. Now conversely, by the observations in the previous paragraph, any three such approximations will determine $\leq 1$ points in $S\cap D$. Hence we can bound the number of points in $S\cap D$ by the cube of the number of approximations in a disc of radius $\ll \frac{1}{M}$.

  By lemma 4, each of our approximations to a point in $S$ will have a disc of radius $\gg \frac{1}{M}$ in which no other approximation will lie, here the relevant $Q$ is $\ll M^{2/3}$. Hence the number of our approximations in any disc of radius $\ll \frac{1}{M}$ is $\ll 1$, and the theorem follows.
\end{proof}

\section{Large Sieve Type Inequality}
\label{sec:largesieve}

Almost as a corollary to theorem 2, we can deduce the following large sieve type inequality,
\begin{theorem}
  \label{theorem:largesieve}
  For any sequence of complex numbers $a_{k,l}$ supported in positive integers $k\leq K$ and $l \leq L$, we have
  \begin{equation}
    \label{eq:54}
    \begin{aligned}
      \sum_{M < m \leq 2M} \sum_{\nu^3 \equiv 2 (m)} \left| \sum_k \sum_l a_{k,l} e\left( \frac{k\nu + l \nu^2}{m} \right) \right|^2 \\
      \ll (M + K)(M+ L) \sum_k\sum_l |a_{k,l}|^2.
    \end{aligned}
  \end{equation}

\end{theorem}

Before getting to the proof, we remark that the inequality is optimal, up to the implied constant, if $K$ and $L$ are both $\geq M$, since in this case the right hand side is
\begin{equation}
  \label{eq:55}
  \ll KL \sum_k\sum_l |a_{k,l}|^2 ,
\end{equation}
which is the size of just one term of the $m$ and $\nu$ sums if we choose
\begin{equation}
  \label{eq:56}
  a_{k,l} = e\left( \frac{-k\nu_0 + -l\nu_0^2}{m_o} \right)
\end{equation}
for a fixed $m_0$ and $\nu_0$. But also, since the length of the sum over $m$ and $\nu$ is $\ll M$ and whence the trivial bound from Cauchy's inequality is
\begin{equation}
  \label{eq:57}
  KLM \sum_k\sum_l |a_{k,l}|^2,
\end{equation}
our theorem gives worse than trivial in the regime $KL \leq M$. 

\begin{proof}
  Utilizing the duality principle, we see that it is enough for our purpose to prove that for any sequence of complex numbers $b_{m,\nu}$,
  \begin{equation}
    \label{eq:58}
    \begin{aligned}
      \sum_{k\leq K}\sum_{l\leq L} \left| \sum_{M < m \leq 2M} \sum_{\nu^3 \equiv 2 (m)} b_{m,\nu} e\left( \frac{k\nu + l\nu^2}{m} \right) \right|^2 \quad \\
      \ll (M + K)(M + L) \sum_{M < m\leq 2M} \sum_{\nu^3 \equiv 2 (m)} |b_{m,\nu}|^2.
    \end{aligned}
  \end{equation}

  Let $f: \mathbb{R} \to \mathbb{R}$ be a smooth function such that $f(x)\geq 0$ for all $x$, $f(x) \geq 1$ for $0 \leq x \leq 1$, and $\hat{f}$, the Fourier transform of $f$, is compactly supported. Then the left hand side of (\ref{eq:58}) is
  \begin{equation}
    \label{eq:59}
    \leq \sum_k f\left(\frac{k}{K}\right) \sum_l f\left( \frac{l}{L}\right) \left| \sum_{M < m \leq 2M} \sum_{\nu^3 \equiv 2 (m)} b_{m,\nu} e\left( \frac{k\nu + l\nu^2}{m} \right) \right|^2.
  \end{equation}

  Expanding out the square, (\ref{eq:59}) becomes
  \begin{equation}
    \label{eq:60}
    \underset{\substack{ M < m \leq 2M \\ \nu^3 \equiv 2 (m)}}{\sum\sum} \underset{\substack{ M < m_1 \leq 2M \\ \nu_1^3 \equiv 2 (m_1)}}{\sum\sum} b_{m,\nu} \overline{b}_{m_1,\nu_1} \mathcal{B}(m,\nu,m_1,\nu_1)\mathcal{B}'(m,\nu,m_1,\nu_1),
  \end{equation}
  where
  \begin{equation}
    \label{eq:61}
    \mathcal{B}(m,\nu,m_1,\nu_1) = \sum_k f\left(\frac{k}{K}\right) e\left( k \left(\frac{\nu}{m} - \frac{\nu_1}{m_1}\right)\right),
  \end{equation}
  and
  \begin{equation}
    \label{eq:62}
    \mathcal{B}'(m,\nu,m_1,\nu_1) = \sum_l f\left(\frac{l}{L}\right) e\left(l\left( \frac{\nu^2}{m} - \frac{\nu_1^2}{m_1}\right)\right).
  \end{equation}

  Applying Poisson summation to (\ref{eq:61}), we have
  \begin{equation}
    \label{eq:63}
    \mathcal{B}(m,\nu,m_1,\nu_1) = K \sum_k \hat{f} \left( K\left(k - \left(\frac{\nu}{m} - \frac{\nu_1}{m_1} \right)\right)\right).
  \end{equation}
  Now, by the compact support of $\hat{f}$, only $k$ for which
  \begin{equation}
    \label{eq:64}
    \left| k - \left( \frac{\nu}{m} - \frac{\nu_1}{m_1} \right)\right| \ll \frac{1}{K}
  \end{equation}
  will contribute to the sum in (\ref{eq:63}). If $K\gg 1$, then at most one $k$ will appear, and even then, only when
  \begin{equation}
    \label{eq:65}
    \left|\left| \frac{\nu}{m} - \frac{\nu_1}{m_1}\right|\right| \ll \frac{1}{K},
  \end{equation}
  where we use $|| \cdot ||$ to denote the distance to the nearest integer, that is the metric on $\mathbb{R}/ \mathbb{Z}$. Hence for $K \gg 1$, we have
  \begin{equation}
    \label{eq:66}
    \mathcal{B}(m,\nu,m_1,\nu_1) \ll K \mathbbm{1}_{ \left|\left| \frac{\nu}{m} - \frac{\nu_1}{m_1}\right|\right| \ll \frac{1}{K}}.
  \end{equation}
  In fact, this bound clearly works for all $K$, perhaps by adjusting the implied constants.

  By the same reasoning, we have the similar bound for $\mathcal{B}'$,
  \begin{equation}
    \label{eq:67}
    \mathcal{B}'(m,\nu,m_1,\nu_1) \ll L \mathbbm{1}_{\left|\left| \frac{\nu^2}{m} - \frac{\nu_1^2}{m_1} \right|\right| \ll \frac{1}{L}}. 
  \end{equation}
  Hence the left hand side of (\ref{eq:58}) is
  \begin{equation}
    \label{eq:68}
    \ll KL \underset{ \substack{ M < m, m_1 \leq 2M \\ \nu^3 \equiv 2(m),\ \nu_1^3 \equiv 2 (m_1) \\ \left|\left| \frac{\nu}{m} - \frac{\nu_1}{m_1}\right|\right| \ll \frac{1}{K},\ \left|\left| \frac{\nu^2}{m} - \frac{\nu_1^2}{m_1} \right|\right| \ll \frac{1}{L}}}{\sum\sum\sum\sum} |b_{m,\nu} b_{m_1,\nu_1}|.
  \end{equation}

  Applying $|b_{m,\nu} b_{m_1,\nu_1}| \leq \frac{1}{2} |b_{m,\nu}|^2 + \frac{1}{2} |b_{m_1,\nu_1}|^2 $ and exploiting the symmetry between $m, \nu$ and $m_1, \nu_1$, we see that (\ref{eq:68}) is
  \begin{equation}
    \label{eq:69}
    \leq KL \sum_{M < m \leq 2M} \sum_{\nu^3 \equiv 2 (m)} |b_{m,\nu}|^2 \underset{\substack{ M < m_1 \leq 2M,\ \nu_1^3 \equiv 2 (m_1) \\ \left( \frac{\nu_1}{m_1}, \frac{\nu_1^2}{m_1} \right) \in R_{m,\nu}}}{\sum\sum} 1,
  \end{equation}
  where $R_{m,\nu}$ is a $O\left( \frac{1}{K} \right) \times O\left( \frac{1}{L} \right)$ rectangle in $\mathbb{R}^2 / \mathbb{Z}^2$. 

  We can cover the rectangle $R_{m,\nu}$ by $\ll \left( \frac{M}{K} + 1\right) \left( \frac{M}{L} + 1\right)$ discs of radius $\frac{1}{M}$, and in each of these discs there are $\ll 1$ points $\left( \frac{\nu_1}{m_1}, \frac{\nu_1^2}{m_1}\right)$, according to theorem 2. We have that (\ref{eq:69}) is then
  \begin{equation}
    \label{eq:70}
    \ll KL \left( \frac{M}{K} + 1\right) \left( \frac{M}{L} + 1 \right) \sum_{M < m \leq 2M} \sum_{\nu^3 \equiv 2 (m)} |b_{m,\nu}|^2,
  \end{equation}
  from which (\ref{eq:59}), and hence theorem 3, follows.
\end{proof}

\bibliographystyle{plain}
\bibliography{./references}

\end{document}